\documentclass[12pt]{article}
\usepackage{amsmath,amssymb,amsthm,fullpage}

\title{The AGM Simple Pendulum}

\author{Mark B. Villarino\\
Depto.\ de Matem\'atica, Universidad de Costa Rica,\\
2060 San Jos\'e, Costa Rica}

\date{\today}

\theoremstyle{plain}
\newtheorem{thm}{Theorem}[section]    
\newtheorem{prop}[thm]{Proposition}   
\newtheorem{cor}[thm]{Corollary}      

\theoremstyle{definition}
\newtheorem{defn}[thm]{Definition}    
\newtheorem{exmp}[thm]{Example}       

\numberwithin{equation}{section}

\makeatletter
\def\section{\@startsection{section}{1}{\z@}{-3.5ex plus -1ex minus
			  -.2ex}{2.3ex plus .2ex}{\large\bf}}
\def\subsection{\@startsection{subsection}{2}{\z@}{-3.25ex plus -1ex
			  minus -.2ex}{1.5ex plus .2ex}{\normalsize\bf}}
\renewcommand{\@dotsep}{200} 
\makeatother

\renewcommand{\geq}{\geqslant}  
\renewcommand{\leq}{\leqslant}  

\newcommand{\hideqed}{\renewcommand{\qed}{}} 

\newcommand{\nth}{\ensuremath{n^{\mathrm{th}}}} 

\newcommand{\word}[1]{\quad\mbox{#1}\quad} 

\begin{document}

\maketitle

\begin{abstract}
We present a self-contained development of \textsc{Gauss}'
Arithmetic-Geometric Mean (AGM) and the work of the great british number theorist  \textsc{A. E. Ingham}
who obtained rigorous error bounds for the AGM's approximations to the
period of a simple pendulum. Moreover we discuss the relation of complex multiplication to the AGM.
\end{abstract}



\section{Introduction}

One of the most celebrated problems in the classical dynamics of
particles is the computation of the period of the simple pendulum. The
nonlinear differential equation which models the pendulum's motion
appears in numerous physical problems and the exact formula for the
period of a single oscillation is given by a complete elliptic
integral of the first kind. But, in 1834, \textsc{Joseph Liouville}
\cite{Liou} proved a justly famous theorem which implies that such an
integral cannot be evaluated by any finite combination of elementary
functions. Therefore, the calculation of the period must be carried
out by suitable approximative formulas. This had been recognized long
before, and in 1747, \textsc{Daniel Bernoulli} \cite{PT} published the
first such approximation. Since then an enormous literature has arisen
around the problem of finding a good approximation to the period and
research continues unabated to this very day.

The authors of these approximations show great dexterity and ingenuity
in their derivations and use a variety of techniques to obtain them.
However, virtually NONE of them offers a rigorous error analysis. That
is to say, inequalities on the upper bound for the error, which shows
how \emph{good} the approximation is, and on the lower bound for the
error, which shows how \emph{bad} the approximation is. (See
\textsc{Thurston} \cite{Thur}). Most of the authors do include
numerical studies of the accuracy of their approximations and some
even include a few order-of-magnitude asymptotics. But those with
rigorous error bounds are few and far between.

Recent interest has concentrated in \textsc{Gauss}'
Arithmetic-Geometric Mean (AGM) algorithm \cite{AGM} because of its
high rate of convergence. In 2008, \textsc{Claudio G. Carvalhaes} and
\textsc{Patrick Suppes} \cite{CS1} published a very interesting and
detailed presentation of the AGM and its application to the
approximation of the period. They also presented an elegant
interpretation of the AGM recurrence formula as a method of
\emph{renormalizing} the pendulum in the sense that it replaces the
original pendulum with another one with the same period, but longer
length and smaller amplitude. This interpretation was already known to
\textsc{Greenhill} \cite{Gre} in the late 1800's, who showed its deep relation to the modern theory of complex multiplication, but has been
woefully neglected till recently. However, their paper, too, fails to
offer any rigorous error analysis, although the numerical studies of
the error are extremely interesting and merit study.

It is unfortunate that none of the authors cites the marvelous
investigations of the great British number theorist \textsc{A. E.
Ingham} which \textsc{L. A. Pars} describes in his monumental 665-page
standard work \cite{Pars1}, which was published almost 50 years ago in
1965. \textsc{Ingham} not only obtains the formulas of
\textsc{Carvalhaes} and \textsc{Suppes} but also obtains rigorous
error estimates, both in excess and in defect. It is beyond question
that Ingham's work deserves to be better known.

So, our paper is organized as follows. To make it as self-contained as
possible, we develop \emph{ab initio} the theory of the AGM including
Gauss' original proof that it converges to the complete elliptic
integral of the first kind. It is difficult to find this anywhere,
today, since the clever method of \textsc{D. J. Newman} \cite{DJN} has
now become fashionable. Then we apply the AGM to the case of the
simple pendulum and we slightly alter the results and proofs of
\textsc{Ingham} so as to obtain a complete error analysis. Finally we
compare our analytical error bounds with the numerical studies of
\textsc{Carvalhaes} and \textsc{Suppes} and show that they virtually
coincide (as they should!). But, and this is the novelty of our paper, our analysis explains the \emph{why}
of their unexplained numerical results.  Our paper shows a beautiful interweaving of classical mechanics and pure mathematics.


\section{The AGM}

We use the Stockholm lectures of \textsc{Vladimir Tkachev} \cite{TK}
in our treatment of the AGM.



\begin{defn}
Let $a\geq 0$ and $b \geq 0$ be two numbers such that $a\geq b\geq 0$
and define the numbers $a_0$ and $b_0$ by
\begin{equation}
\label{a0b0}
a_0 := a,  \qquad  b_0 := b;
\end{equation}
then for $n = 0,1,2,\dots$ define the sequences $\{a_n\}$ and
$\{b_n\}$ by
\begin{equation}
\label{anbn}
a_{n+1} := \frac{a_n + b_n}{2},  \qquad   b_{n+1} := \sqrt{a_n b_n}\,.
\end{equation}
\end{defn}

Note that each $a_{n+1}$ is the \emph{arithmetic mean} of the previous
$a_n$ and $b_n$, while each $b_{n+1}$ is the \emph{geometric mean} of
those same two numbers.

\begin{defn}
One says that the sequences $\{a_n\}$ and $\{b_n\}$ in \eqref{a0b0}
and \eqref{anbn} define the \textbf{arithmetic-geometric mean
algorithm}, which we abbreviate as \textbf{AGM}.
\end{defn}


Now we collect some of the basic properties of the AGM.

\begin{prop}
The following properties of the AGM are valid.
\begin{enumerate}
\item 
The $a_n$'s decrease, the $b_n$'s increase and every $a_n$ is bigger
than every $b_m$. More precisely,
\begin{equation}
\label{abmon}
a_0 \geq a_1 \geq\cdots\geq a_n \geq a_{n+1} \geq\cdots\geq b_{n+1}
\geq b_n \geq\cdots\geq b_1 \geq b_0.
\end{equation}
\item 
\begin{equation}
\label{aminusb}
0 \leq a_n - b_n \leq \frac{a-b}{2^n} \,.
\end{equation}
\item 
The limits
\begin{equation}
\label{AB}
A := \lim_{n\to\infty} a_n  \word{and}  B := \lim_{n\to\infty} b_n
\end{equation}
both exist and they are equal,
\begin{equation}
\label{A=B}
A = B.
\end{equation}
\end{enumerate}
\end{prop}

\begin{proof} \emph{Of \eqref{abmon}}:

Since the square of a real number is always non-negative, it follows
that for $n = 0,1,2,\dots$, $(\sqrt{a_n} - \sqrt{b_n})^2 \geq 0$ and
that there is strict inequality unless $a_n = b_n$, whence we conclude
that the following inequality is valid,
\begin{equation}
\label{agm}
\frac{a_n+b_n}{2} \geq \sqrt{a_nb_n} \,.
\end{equation}
Of course, \eqref{agm} is the famous \emph{arithmetic-geometric mean
inequality} for two numbers. Applying it to $a_{n+1}$ and $b_{n+1}$ we
obtain
\begin{equation}
a_{n+1} \geq b_{n+1}.
\end{equation}
Thus, from $a_{n+1} \geq b_{n+1}$ and $a_n \geq b_n$ we obtain
\begin{equation}
a_n \geq \frac{a_n+b_n}{2} =: a_{n+1} \geq b_{n+1} := \sqrt{a_nb_n}
\geq b_n,
\end{equation}
which is \eqref{abmon}.

\medskip

\emph{Of \eqref{aminusb}}:
\\
{}From $b_{n+1} \geq b_n$ we conclude
$$
a_{n+1} - b_{n+1} \leq a_{n+1} - b_n = \frac{a_n+b_n}{2} - b_n
= \frac{a_n-b_n}{2}
$$
and \eqref{aminusb} follows by induction.

\medskip

\emph{Of \eqref{AB} and \eqref{A=B}}:
\\
By \eqref{abmon} the sequence $\{a_n\}$ decreases monotonically and is
bounded from below by $b_0$, and so $A$ exists. By \eqref{abmon} the
sequence $\{b_n\}$ increases monotonically and is bounded from above
by $a_0$, and so $B$ exists.

Finally, letting $n$ tend to infinity in \eqref{aminusb} and using
\eqref{AB}, we obtain
$$
0 \leq A - B \leq 0
$$ 
and by the ``squeeze'' theorem, we conclude $A = B$.
\end{proof}


Now the following definition makes sense.

\begin{defn}
We define the \textbf{arithmetic-geometric mean}, $M(a,b) \equiv \mu$
of the numbers $a$ and $b$ to be the common limit
\begin{equation}
M(a,b) \equiv \mu := A := \lim_{n\to\infty} a_n
\equiv B := \lim_{n\to\infty} b_n
\end{equation}
of the AGM as applied to the numbers $a$ and $b$.
\end{defn}

\begin{prop}
The geometric mean $b_n$ is a \textbf{closer} approximation to $\mu$ than $a_n$; more precisely
\begin{equation}
\label{gmb}
0 < \frac{\mu - b_n}{a_n - \mu} < 1.
\end{equation}
\end{prop}

\begin{proof}
We observe
$$
\mu < a_{n+1} = \frac{a_n + b_n}{2} \iff 2\mu < a_n + b_n
\iff \mu - b_n < a_n - \mu.
$$
Since $0 < \mu - b_n < a_n - \mu$, we can divide by $a_n - \mu$ to
complete the proof.
\end{proof}


\section{Gauss' theorem on elliptic integrals}

The following theorem gives a hint of the depth of the mathematics
involved in the AGM. It is the only theorem Gauss published on the
algorithm and appears in a paper on secular variations (!) published in  1818  \cite{AGM}. But,
it seems that he already had a proof in 1799 \cite{Cox}:

\begin{thm}\label{Gauss}
Let $a$ and $b$ be positive real numbers. Then
\begin{equation}
\label{EI}
\boxed{ \frac{1}{M(a,b)} = \frac{2}{\pi} \int_0^{\frac{\pi}{2}}
\frac{d\phi}{\sqrt{a^2 \cos^2\phi + b^2 \sin^2\phi}} }\,.
\end{equation}
\end{thm}

The integral \eqref{EI} is a \emph{complete elliptic integral of the
first kind} and, as we have already seen \cite{Liou}, cannot be
evaluated in finite terms with elementary functions. In the next
section we will see its relationship to the simple pendulum.

Before we enter into the details of Gauss' proof, we introduce some
notation and separate out the fundamental technical step.

Let
\begin{equation}
\label{Iab}
I(a,b):=\int_0^{\frac{\pi}{2}} 
\frac{d\phi}{\sqrt{a^2 \cos^2\phi + b^2 \sin^2\phi}} \,.
\end{equation}
Then, we have to prove that 
\begin{equation}
I(a,b) = I(a_1,b_1) = I(a_2,b_2) = I(a_3,b_3) = \cdots
\end{equation}
since we can then conclude that
\begin{equation}
I(a,b) = \lim_{n\to\infty} I(a_n,b_n) = I(\mu,\mu) = \frac{\pi}{2\mu}
\end{equation}
which, after multiplying by $\dfrac{2}{\pi}$, is precisely~\eqref{EI}.  

In order to conclude
\begin{equation}
\label{Iab1}
\lim_{n\to\infty} I(a_n,b_n)
= I\bigl( \lim_{n\to\infty} a_n, \lim_{n\to\infty} b_n \bigr)
= I(\mu,\mu)
\end{equation}
we have to prove that \emph{we can interchange the limit and the
integral signs}. For this, it is sufficient to prove:

\begin{prop}
The sequence
$\biggl\{\dfrac{1}{\sqrt{a_n^2\cos^2\phi + b_n^2\sin^2\phi}}\biggr\}$,
$n = 0,1,2,\dots$, converges \textbf{uniformly} to $\dfrac{1}{\mu}$.
\end{prop}

\begin{proof}
That means given any $\epsilon > 0$ we must prove there exists
positive number $N(\epsilon)$, which is \emph{independent} of the
variable $\phi$, such that the following implication is true:
\begin{equation}
\label{unifc}
n > N(\epsilon) \implies  \biggl|
\frac{1}{\sqrt{a_n^2\cos^2\phi +b_n^2\sin^2\phi}}
- \frac{1}{\mu} \biggr| < \epsilon.
\end{equation}
However, the identity $\cos^2\phi + \sin^2\phi = 1$ as well as the
inequalities \eqref{abmon} and
$$
b_n \leq \sqrt{a_n^2 \cos^2\phi + b_n^2 \sin^2} \leq a_n
$$
show us that
$$
-(a_n - b_n) = b_n - a_n < b_n - \mu
< \sqrt{a_n^2 \cos^2\phi + b_n^2 \sin^2\phi} - \mu < a_n - \mu
< a_n - b_n,
$$
that is,
\begin{equation}
\label{aminusb1}
\biggl| \sqrt{a_n^2 \cos^2\phi + b_n^2 \sin^2\phi} - \mu \biggr|
< a_n - b_n < \frac{a-b}{2^n}
\end{equation}
where we applied \eqref{aminusb} in the last inequality. Now,
$$
\biggl| \frac{1}{\sqrt{a_n^2 \cos^2\phi + b_n^2 \sin^2\phi}}
- \frac{1}{\mu} \biggr|
= \biggl| \frac{\sqrt{a_n^2 \cos^2\phi + b_n^2 \sin^2\phi} - \mu}
{\mu \cdot \sqrt{a_n^2 \cos^2\phi + b_n^2 \sin^2\phi}} \biggr|
< \frac{a-b}{2^n} \frac{1}{b^2}
$$ 
by \eqref{aminusb1} and \eqref{abmon}. For the implication
\eqref{unifc} to be true, it is sufficient that the following
inequality be true:
$$
\frac{a - b}{2^n} \frac{1}{b^2} < \epsilon
\iff 2^n > \frac{a - b}{b^2\epsilon}
\iff n > \frac{\ln\bigl( \frac{a-b}{b^2\epsilon} \bigr)}{\ln 2} \,,
$$
that is, the choice 
\begin{equation}
N(\epsilon) := \frac{\ln\bigl( \frac{a-b}{b^2\epsilon} \bigr)}{\ln 2}
\end{equation}
proves the truth of the implication \eqref{unifc}, and that,
therefore, we can interchange the limit and integral signs
in~\eqref{Iab1}.
\end{proof}


Gauss' original proof is based on the following \emph{change of
variable in the integral} $I(a,b)$: we introduce a new variable,
$\phi'$ instead of $\phi$ by the formula:
\begin{equation}
\label{phi'}
\boxed{\sin\phi =: \frac{2a\sin\phi'}{a + b + (a - b) \sin^2\phi'}}\,.
\end{equation}

\begin{prop}
\label{pr:phi'}
Under the mapping \eqref{phi'} the interval
$0 \leq \phi' \leq \frac{\pi}{2}$ corresponds bijectively to the 
interval $0 \leq \phi \leq \frac{\pi}{2}$.
\end{prop}

\begin{proof}
Define the function
\begin{equation}
f(t) := \frac{2at}{a + b + (a - b)t^2} \,.
\end{equation}
Then 
\begin{equation}
f'(t) = 2a\frac{a + b - (a - b)t^2}{\{a + b + (a - b)t^2\}^2}
\geq \frac{2ab}{\{a + b + (a - b)t^2\}^2} > 0,
\end{equation}
which proves that $f(t)$ is \emph{increasing} on $[0,1]$. Moreover,
$$
f(0) = 0,  \qquad  f(1) = 1,
$$
which shows that $f(t)$ maps $[0,1]$ bijectively onto itself. This
completes the proof.
\end{proof}


\begin{proof}[Proof of Gauss' theorem on elliptic integrals]

Gauss, himself \cite{AGM}, first states Theorem~\ref{Gauss}.
Then he blithely asserts
\begin{quote}

 ``\emph{Evolutione autem rite facta,
invenitur esse}...'' 

\end{quote}
which translates to 
\begin{quote}
``\emph{After the development
has been made correctly, it will be seen (that)\dots}
\begin{equation}
\label{EI1}
\frac{d\phi}{\sqrt{a^2 \cos^2\phi + b^2 \sin^2\phi}}
= \frac{d\phi'}{\sqrt{a_1^2 \cos^2\phi' + b_1^2 \sin^2\phi'}}\,.
\text{''}
\end{equation}

\end{quote}
(We have changed Gauss' notation: he writes $m,n,m',n',T,T'$ in place
of our $a,b,a_1,b_1,\phi,\phi'$, respectively.) This, of course, is
the step
\begin{equation}
\label{Ia1b1}
I(a,b) = I(a_1,b_1)
\end{equation}
in the notation~\eqref{Iab}.

We will show how the development is ``made correctly'' (!).

\medskip

\noindent \emph{Claim 1:}
\begin{equation}
\label{claim1}
\cos\phi
= \frac{2\cos\phi' \sqrt{a_1^2 \cos^2\phi' + b_1^2 \sin^2\phi'}}
{a + b + (a - b)\sin^2\phi'} \,.
\end{equation}

\begin{proof}
By \eqref{phi'} and \eqref{anbn},
\begin{align*}
\cos^2\phi  = 1 - \sin^2\phi
\\
&= 1 - \frac{4a^2 \sin^2\phi'}{\{(a + b) + (a - b) \sin^2\phi'\}^2}
\\
&= \frac{(a + b)^2 + 2(a^2 - b^2) \sin^2\phi'
+ (a - b)^2 \sin^4\phi' - 4a^2 \sin^2\phi'}
{\{(a + b) + (a - b) \sin^2\phi'\}^2}
\\
&= \frac{4a_1^2 - 4(2a_1^2 - b_1^2) \sin^2\phi'
+ 4(a_1^2 - b_1^2) \sin^4\phi'}{\{(a + b) + (a - b) \sin^2\phi'\}^2}
\\
&= \frac{4a_1^2 \cos^4\phi' + 4b_1^2 \sin^2\phi' \cos^2\phi'}
{\{(a + b) + (a - b) \sin^2\phi'\}^2}
\end{align*}
and factoring out $4\cos^2\phi'$ and taking the square root of both
sides gives us~\eqref{claim1}.
\end{proof}

\noindent \emph{Claim 2:}
\begin{equation}
\label{Claim2}
\sqrt{a^2 \cos^2\phi + b^2 \sin^2\phi}
= a\,\frac{(a + b) + (a - b)\sin^2\phi'}
{(a + b) - (a - b)\sin^2\phi'} \,.
\end{equation}

\begin{proof}
By \eqref{claim1}, \eqref{phi'} and \eqref{anbn}, we obtain

\begin{align*}
a^2 \cos^2\phi + b^2 \sin^2\phi &= a^2 \biggl\{
\frac{2 \cos\phi' \sqrt{a_1^2 \cos^2\phi' + b_1^2 \sin^2\phi'}}
{a + b + (a - b) \sin^2\phi'} \biggr\}^2
+ \frac{4a^2b^2 \sin^2\phi'}{\{(a + b) + (a - b) \sin^2\phi'\}^2} 
\\
&= \frac{4a^2 \cos^2\phi' (a_1^2 \cos^2\phi' + b_1^2 \sin^2\phi')
+ 4a^2b^2 \sin^2\phi'}{\{(a + b) + (a - b)\sin^2\phi'\}^2}
\\
&= 4a^2 \frac{a_1^2(1 - \sin^2\phi')^2
+ b_1^2\sin^2\phi'(1 - \sin^2\phi') + b^2\sin^2\phi'}
{\{(a + b) + (a - b) \sin^2\phi'\}^2}
\\
&= a^2 \frac{(a + b)^2(1 - \sin^2\phi')^2 
+ 4ab \sin^2\phi'(1 - \sin^2\phi') + (a - b)^2 \sin^4\phi'}
{\{(a + b) + (a - b) \sin^2\phi'\}^2}
\\
&= a^2\frac{(a + b)^2 - 2(a - b)(a + b) \sin^2\phi' + 4b^2 \sin^2\phi'}
{\{(a + b) + (a - b) \sin^2\phi'\}^2}
\\
&= \biggl\{ a\,\frac{(a + b) + (a - b)\sin^2\phi'}
{(a + b) - (a - b) \sin^2\phi'} \biggr\}^2
\end{align*}
and taking the square root of both sides gives us~\eqref{Claim2}.
\end{proof}

Now we can complete the proof of Gauss' theorem. We take the
differential of the left hand side of \eqref{phi'}: we obtain
$$
\cos\phi \,d\phi
= \frac{2\cos\phi' \sqrt{a_1^2 \cos^2\phi' + b_1^2 \sin^2\phi' }}
{a + b + (a - b) \sin^2\phi'} \,d\phi,
$$
where we applied \eqref{claim1}. 

Taking the differential of the right side of \eqref{phi'} we get
$$
d \biggl\{ \frac{2a\sin\phi'}{a + b + (a - b)\sin^2\phi'} \biggr\}
= \frac{2a\cos\phi'\{(a + b) - (a - b) \sin^2\phi'\}}
{\{(a + b) + (a - b) \sin^2\phi'\}^2} \,d\phi'.
$$

Equating the right hand side of the previous two equations we and
using \eqref{Claim2} we obtain
\begin{align*}
\frac{2\cos\phi' \sqrt{a_1^2 \cos^2\phi' + b_1^2 \sin^2\phi'}}
{a + b + (a - b) \sin^2\phi'} \,d\phi
&= \frac{2a\cos\phi' \{(a + b) - (a - b) \sin^2\phi'\}}
{\{(a + b) + (a - b) \sin^2\phi'\}^2} \,d\phi' 
\\
\implies  \frac{d\phi}{\sqrt{a^2 \cos^2\phi  + b^2 \sin^2\phi}}
&= \frac{a\{(a + b) - (a - b) \sin^2\phi'\}}
{\{(a + b) + (a - b) \sin^2\phi'\}^2} \,d\phi'
\cdot \frac{\frac{\{(a+b)-(a-b)\sin^2\phi'\}^2}
{a\{(a+b)-(a-b)\sin^2\phi'\}}}
{\sqrt{a_1^2 \cos^2\phi' + b_1^2 \sin^2\phi'}}
\\ 
&= \frac{d\phi'}{\sqrt{a_1^2 \cos^2\phi' + b_1^2 \sin^2\phi'}} \,.
\end{align*}
This completes the proof of \eqref{EI1}, therefore of \eqref{Ia1b1},
and therefore, of Gauss' theorem.
\end{proof}


\section{The Simple Pendulum}

First, we define the dynamical system. It is an idealization of a real
pendulum.

\begin{defn}
The \textbf{simple pendulum} consists of a particle which is
constrained to move without friction on the circumference of a
vertical circle and which is acted upon only by gravity. We describe
it mechanically as follows:
\begin{itemize}
\item
a massless inextensible rigid rod has a point-mass attached to one 
end;
\item
the rod is suspended from a frictionless pivot;
\item
when the point-mass is given an initial push perpendicular to the rod,
it will swing back and forth in one vertical plane and with a constant
amplitude;
\item
there is no air resistance.
\end{itemize}
\end{defn}

The following properties of the simple pendulum are readily available
in numerous textbooks. For example, the standard work of \textsc{Pars}
\cite{Pars1}.

The \emph{nonlinear differential equation} which models the motion of
the simple pendulum is
\begin{equation}
\label{DE}
\frac{d^2\theta}{dt^2} + \frac{g}{l} \sin\theta = 0,
\end{equation}
where $g$ is the acceleration due to gravity, $l$ is the length of the
pendulum, and $\theta(t)$ is the angular displacement, at time $t$, of
the pendulum measured positively (counter-clockwise) from the
vertical equilibrium position.

The \emph{Period} of the pendulum, $T$, is the time taken by a double
oscillation, to and fro, and is given by the following famous formula
\begin{equation}
\label{Period}
T = 4\sqrt{\frac{l}{g}} \int_0^{\frac{\pi}{2}} 
\frac{d\phi}{\sqrt{1 - \sin^2\frac{\alpha}{2} \sin^2\phi}}
\end{equation}
where $\alpha$ is the maximum angular displacement of the pendulum.

The formula \eqref{Period} shows that the period is proportional to a
product of a function of the length, $l$, alone, and the maximum
angular amplitude, $\alpha$, alone. That is, there is already a
``separation of variables'' in the formula for the period.

The integral in \eqref{Period} is a \emph{complete elliptic integral
of the first kind} and, as we already noted, cannot be evaluated by
any finite combination of elementary functions. So we must find
suitable \emph{approximative formulas} for~$K$.

It is customary to write $k := \sin \frac{\alpha}{2}$ so that the
integral in \eqref{Period} is
\begin{equation}
\label{K}
K(k) = \int_0^{\frac{\pi}{2}}\frac{d\phi}{\sqrt{1 - k^2\sin^2\phi}}\,.
\end{equation}
The quantity $k$ is called the \emph{modulus} of $K$ and
$\frac{\alpha}{2}$ is called the \emph{modular angle}. In our case,
the modular angle is one-half of the maximum angular displacement of
the pendulum, and we write (with an abuse of notation)
$K(k) \equiv K(\alpha).$ The \emph{complimentary modulus}, $k' \geq 0$
is defined by $k^2 + k'^2 = 1$ and the corresponding complete elliptic
integral of the first kind is $K(k')\equiv K'$.

Expanding \eqref{K} by the binomial theorem and integrating term by
term, we obtain
\begin{equation}
K(k) = \frac{\pi}{2} \biggl\{ 1 + \sum_{n=1}^\infty \biggl[
\frac{1 \cdot 3 \cdot 5 \cdots (2n - 1)}{2 \cdot 4 \cdot 6 \cdots 2n}
\biggr]^2 k^{2n} \biggr\}.
\end{equation}
This gives us the fundamental theorem:


\begin{thm}
\label{Per}
The period of a simple pendulum of length $l$, oscillating through an
angle~$2\alpha$, is equal to
\begin{align*}
T &= 4\sqrt{\frac{l}{g}} \cdot K(k)
\\
&= 2\pi \sqrt{\frac{l}{g}} \biggl\{ 1 + \Bigl(\frac{1}{2}\Bigr)^2
\Bigl(\sin\frac{\alpha}{2}\Bigr)^2
+ \Bigl(\frac{1\cdot 3}{2\cdot 4}\Bigr)^2 
\Bigl(\sin\frac{\alpha}{2}\Bigr)^4
+ \Bigl(\frac{1\cdot 3\cdot 5}{2\cdot 4\cdot 6}\Bigr)^2
\Bigl(\sin\frac{\alpha}{2}\Bigr)^6 + \cdots \biggr\}
\\
&= 2\pi \sqrt{\frac{l}{g}} \biggl\{ 1 + \frac{1}{16}\alpha^2
+ \frac{11}{3072}\alpha^4 + \frac{173}{737280}\alpha^6
+ \frac{22931}{1321205760}\alpha^8 + \cdots \biggr\}.
\end{align*} 
\qed
\end{thm}

The last formula in Theorem~\ref{Per} comes from substituting the
\textsc{MacLaurin} expansion of $\sin\frac{\alpha}{2}$ into the
previous series, and rearranging in increasing powers of~$\alpha$.

Thus, if a pendulum swinging through an angle of $2\alpha$ makes $N$
beats a day, and if $\alpha$ is increased by $\delta\alpha$, then the
formula $N \cdot \frac{T}{2} = 86400$, where $86400$ is the number of
seconds in a day, shows that \emph{the pendulum will lose
$$
43200 \sqrt{\frac{g}{l}} \biggl\{ \frac{1}{K(\alpha)}
- \frac{1}{K(\alpha + \delta\alpha)} \biggr\}
$$
beats a day}.

For example, a pendulum, which beats seconds when swinging through an
angle of $6^o$ will \emph{lose} about $11\frac{1}{2}$ seconds a day if
made to swing through $8^\circ$, and about $26\frac{1}{3}$ seconds a day
if made to swing through $10^\circ$~\cite{Gre}.

If we truncate the previous series expansions for the period, we obtain the following approximative formulas:
\begin{align}
T &\approx 2\pi \sqrt{\frac{l}{g}} 
\nonumber \\
&\approx 2\pi \sqrt{\frac{l}{g}} \biggl\{ 1
+ \Bigl(\frac{1}{2}\Bigr)^2 \Bigl(\sin\frac{\alpha}{2}\Bigr)^2 \biggr\}
\nonumber \\ 
&\approx 2\pi \sqrt{\frac{l}{g}} \biggl\{ 1 + \frac{1}{16} \alpha^2
\biggr\}.
\end{align}

The first formula, 
\begin{equation}
\label{Huy}
T \approx  2\pi \sqrt{\frac{l}{g}} \equiv T_0
\end{equation}
is the \textsc{Huygens} \emph{formula} or the \emph{small angle
approximation} for the period. It does not contain $\alpha$ and gives
an approximation which is \emph{independent} of the period. Indeed, it
is the formula for the period of \emph{simple harmonic motion}, (SHM),
realized by a particle travelling in a circular path of with constant
angular velocity $\sqrt{\dfrac{g}{l}}$.

Just how accurate is the Huygens formula? It seems worthwhile to cite
the lower bound found by \textsc{Thurston}~\cite{Thur} and the upper
bound found by \textsc{Pars}~\cite{Pars1}.

\begin{cor}[Pars--Thurston]
The Huygens small-angle approximation satisfies
\begin{equation}
\frac{\frac{\alpha}{2}} {\sin\frac{\alpha}{2}}
\leq \frac{T}{2\pi\sqrt{\frac{l}{g}}}
\leq \sqrt{\frac{\alpha}{\sin\alpha}}
\end{equation}
for $0 \leq \alpha \leq \frac{\pi}{2}$.
\qed
\end{cor}

As Thurston points out, if he were to use the Huygens formula to
adjust the length of his grandfather clock, which has an amplitude of
$5^\circ$, to beat seconds, the error bounds show that the clock would
lose between $4$ and $8$ minutes per week.

These same bounds show that the Huygens formula is accurate to within
$1\%$ of the true period $T$ for $\alpha$ smaller than
about~$14^\circ$.

The second formula, 
\begin{equation}
\label{P2}
T \approx 2\pi \sqrt{\frac{l}{g}} \biggl\{ 1
+ \Bigl(\frac{1}{2}\Bigr)^2 \Bigl(\sin\frac{\alpha}{2}\Bigr)^2 \biggr\}
\end{equation}
tells us that in the correction for the amplitude of a swing, the
period must be increased by the fraction
$\frac{1}{4} \sin^2\frac{\alpha}{2}$ of itself. Thus, if a pendulum
swinging through an angle of $2\alpha$ makes $N$ beats a day, and if
$\alpha$ is increased by $\delta \alpha$, \emph{the pendulum will lose
approximately
$\bigl( \frac{N}{8} \cdot \sin\alpha\cdot\delta\alpha \bigr)$ beats per
day}~\cite{Bow}.

The last formula
\begin{equation}
\label{DBer}
T \approx 2\pi \sqrt{\frac{l}{g}}
\biggl\{ 1 + \frac{1}{16} \alpha^2 \biggr\}
\end{equation}
is due to \textsc{Daniel Bernoulli} and is, historically, the
\emph{first} published correction term~\cite{PT} to the Huygens
formula~\eqref{Huy}.


\section{The AGM approximations to the period}

The previous section shows that the problem of finding an
approximative formula for the period $T$ of the simple pendulum
reduces to the problem of approximating the complete elliptic integral
$K(k)\equiv K(\alpha)$.

If we take $a := 1$ and $b := k' = \cos\frac{\alpha}{2}$ in the
formula \eqref{Iab} for Gauss' integral $I(a,b)$, we obtain
\begin{equation}
I\Bigl(1, \cos\frac{\alpha}{2}\Bigr) = K(k)
= \int_0^{\frac{\pi}{2}} 
\frac{d\phi}{\sqrt{1 - \sin^2\frac{\alpha}{2} \sin^2\phi}}
\equiv \int_0^{\frac{\pi}{2}} \frac{d\phi}{\sqrt{1-k^2\sin^2\phi}} \,.
\end{equation}

Applying the AGM to this choice of $a$ and $b$ we obtain the following
sequences:
\begin{align*}
a_0  &= 1 
& b_0 &= \cos\frac{\alpha}{2}  
\\
a_1 &= \frac{1}{2} \Bigl(1 + \cos\frac{\alpha}{2} \Bigr)
= \cos^2\frac{\alpha}{4}
& b_1 &= \Bigl( \cos\frac{\alpha}{2} \Bigr)^{\frac{1}{2}}
\\
a_2 &= \frac{1}{4} \biggl\{
1 + \Bigl( \cos\frac{\alpha}{2} \Bigr)^{\frac{1}{2}} \biggr\}^2
& b_2 &= \cos\frac{\alpha}{4}
\Bigl( \cos\frac{\alpha}{2} \Bigr)^{\frac{1}{4}}
\\
&= \frac{1}{2} \biggl\{ \cos^2\frac{\alpha}{4} 
+ \Bigl( \cos\frac{\alpha}{2} \Bigr)^{\frac{1}{2}} \biggr\}
\\
a_3 &= \frac{1}{4} \biggl\{ \cos^2\frac{\alpha}{4}
+ \Bigl( \cos\frac{\alpha}{2} \Bigr)^{\frac{1}{2}} \biggr\}^2
& b_3 &= \frac{1}{2} \biggl\{
1 + \Bigl( \cos\frac{\alpha}{2}\Bigr)^{\frac{1}{2}} \biggr\}
\Bigl( \cos\frac{\alpha}{4} \Bigr)^{\frac{1}{2}}
\Bigl( \cos\frac{\alpha}{2} \Bigr)^{\frac{1}{8}}
\\    
\cdots &= \cdots & \cdots &= \cdots
\end{align*}

We can use either $\dfrac{1}{a_n}$ or $\dfrac{1}{b_n}$ as an
approximation to $\dfrac{1}{\mu}$.

In order to discuss the accuracy of these approximations, we recall
some definitions from numerical analysis. See
\textsc{Hildebrand}~\cite{Hil}. Each digit of a number, except zero,
which serves only to fix the position of the decimal point is called a
\emph{significant digit} or a \emph{significant figure} of that
number.

\begin{defn}
\begin{enumerate}
\item[(a)]
If any approximation $\overline{N}$ to a number $N$ has the property
that both $\overline{N}$ and $N$ round to the same set of significant
figures, and if $n$ is the LARGEST integer for which this statement is
true, then $\overline{N}$ is said to \textbf{approximate $N$ to $n$
significant digits}.
\item[(b)]
\begin{equation}
\label{RE}
R(\overline{N}) \equiv \text{\textbf{relative error}}
:= \frac{\text{true value} - \text{approximate value}}
{\text{true value}} \equiv \frac{E(\overline{N})}{N} \,,
\end{equation}
where $E \equiv E(\overline{N})$ is the \textbf{absolute error}.
 \end{enumerate}
\end{defn}

The importance of the \emph{relative} error is shown in the following
result.

\begin{prop}
$\overline{N}$ approximates $N$ to $n$ significant digits if and only
if
\begin{equation}
R(\overline{N}) < \frac{(\frac{1}{2})}{10^n}.
\end{equation}
\qed
\end{prop}

\begin{prop}
If $R \equiv R(\overline{N})$ and
\begin{equation}
\overline{R} \equiv \overline{R}(\overline{N})
:= \frac{E(\overline{N})}{\overline{N}} \,,
\end{equation}
then
\begin{equation}
\label{RR}
\overline{R} = \frac{R}{1 - R} \word{and} 
R = \frac{\overline{R}}{1 + \overline{R}} \,.
\end{equation}
\qed
\end{prop}

Now we are ready to present \textsc{Ingham}'s results.


\begin{thm}[\textsc{Ingham}]
Let $R_n$ be the relative error in the approximation
$\dfrac{1}{\mu} \approx \dfrac{1}{a_n}$ and $r_n$ be the relative
error in the approximation $\dfrac{1}{\mu} \approx \dfrac{1}{b_n}$
taken positively. That is, let
\begin{equation}
\boxed{ \Bigl(\frac{1}{1 + r_n}\Bigr) \cdot \frac{1}{b_n}
:= \frac{1}{\mu} =: \Bigl(\frac{1}{1 - R_n}\Bigr) \cdot \frac{1}{a_n}}
\end{equation}
Then,
\begin{equation}
\label{error}
0 < r_n < R_n < \frac{a_n - b_n}{2a_{n + 1}} \,.
\end{equation}
\end{thm}

\begin{proof}
If we divide the numerator and denominator of \eqref{gmb} by $\mu$, we
see that the relative error, $\overline{R}_n$, in the approximation
$\mu\approx a_n$ is greater than the relative error, $\overline{r}_n$,
in the approximation $\mu \approx b_n$. Substituting the formulas
\eqref{RR} into the inequality $0 < \overline{r}_n < \overline{R}_n$,
some simple algebra leads us to the inequality
\begin{equation}
0 < r_n < R_n.
\end{equation}
Moreover, 
\begin{equation}
R_n = \frac{\frac{1}{\mu} - \frac{1}{a_n}}{\frac{1}{\mu}}
= \frac{a_n - \mu}{a_n} < \frac{a_n - b_{n+1}}{a_n}
= \frac{a_n - b_n}{a_n + b_{n+1}} < \frac{a_n - b_n}{2a_{n+1}}
\end{equation}
where the second equality follows from
$$
a_n^2 - b_{n+1}^2 = a_n(a_n - b_{n+1})
$$
and the second inequality follows from
$$
a_n + b_{n+1} > a_n + b_n = 2a_{n+1}.
\eqno \qed
$$
\hideqed
\end{proof}

\begin{thm}[\textsc{Ingham}]
If $T_0$ denotes the Huygens small-angle approximation to the true
period, $T$, then, for $0 < \alpha < \pi$,
\begin{equation}
\boxed{ \frac{T}{T_0} = \Biggl\{
\frac{2}{1 + ( cos\frac{\alpha}{2})^{\frac{1}{2}}} \Biggr\}^2
\cdot \biggl( \frac{1}{1 - R_2} \biggr)
= \frac{1}{\cos\frac{\alpha}{4} (\cos\frac{\alpha}{2})^{\frac{1}{4}}}
\cdot \biggl( \frac{1}{1 + r_2} \biggr) }
\end{equation}
where
\begin{equation}
\label{Inga2}
0 < r_2 < R_2 < \frac{1}{2^6\cos\frac{\alpha}{2}}
\Bigl( \sin\frac{\alpha}{4} \tan\frac{\alpha}{4} \Bigr)^4.
\end{equation}
\end{thm}

\begin{proof}
If we take $n = 2$ in the error estimate \eqref{error}, and using 
\begin{equation}
8a_{n+1}(a_n - b_n) = (a_{n-1} - b_{n-1})^2,
\end{equation}
we obtain
\begin{align*}
\label{Ing1}
0 < r_2 < R_2 < \frac{a_2 - b_2}{2a_{3}}
&= \frac{1}{2a_3} \frac{(a_1 - b_1)^2}{8a_3}
= \frac{1}{2a_3} \frac{1}{8a_3} \frac{(a_0 - b_0)^4}{(8a_2)^2}
= \frac{\{\frac{1}{2}(a - b)\}^4}{2^6a_3^2a_2^2}
\\
&= \frac{(\sin\frac{\alpha}{4})^8}{2^6a_3^2a_2^2}
< \frac{(\sin\frac{\alpha}{4})^8}{2^6b_2^4}
= \frac{(\sin\frac{\alpha}{4})^8}
{2^6(\cos\frac{\alpha}{4})^4 \cos\frac{\alpha}{2}}
\\
&= \frac{1}{2^6\cos\frac{\alpha}{2}}
\Bigl( \sin\frac{\alpha}{4} \tan\frac{\alpha}{4} \Bigr)^4.
\tag*{\qed}
\end{align*}
\hideqed
\end{proof}

\begin{cor}[\textsc{Ingham}]
If $0 < \alpha \leq \frac{\pi}{2}$, then
\begin{equation}
\label{Ing2}
0 < r_2 < R_2 < \frac{1}{70000}
\end{equation}
and thus the approximation is correct to 5 significant figures in the
worst case, $\alpha = \frac{\pi}{2}$.
\end{cor}

\begin{proof}
If we take $\alpha = \frac{\pi}{2}$ in \eqref{Inga2} and note that
\begin{equation}
\sin\frac{\alpha}{4} = \sin\frac{\pi}{8}
= \frac{\sqrt{2 - \sqrt{2}}}{2}\,,  \qquad
\cos\frac{\alpha}{4} = \cos\frac{\pi}{8}
= \frac{\sqrt{2 + \sqrt{2}}}{2}\,,
\end{equation}
and therefore
\begin{equation}
\frac{\cos^2\frac{\alpha}{4}}{\sqrt{2} + 1}
= \frac{\sin^2\frac{\alpha}{4}}{\sqrt{2} - 1} = \frac{1}{2\sqrt{2}}\,,
\end{equation}
we conclude
\begin{align*}
\frac{1}{2^6\cos\frac{\alpha}{2}}
\Bigl( \sin\frac{\alpha}{4} \tan\frac{\alpha}{4} \Bigr)^4
&= \frac{\sqrt{2}}{2^9} \frac{1}{(\sqrt{2} + 1)^6}
= \frac{\sqrt{2}}{2^9(99 + 70\sqrt{2})}
\\
&= \frac{1}{2^{8\frac{1}{2}}(99 + 70\sqrt{2})}  
\\
&< \frac{1}{2^{8\frac{1}{2}}\cdot 70\cdot 2\sqrt{2}\dots}  \quad
\word{$\{$since}  (99)^2 > (70\sqrt{2})^2 \Rightarrow 99+70\sqrt{2}>70\cdot 2\sqrt{2}\}
\\
&< \frac{1}{2^{10}\cdot 70} = \frac{1}{71680} < \frac{1}{70000} \,.
\\
\end{align*}
Finally, we have shown that
$R_2 < \dfrac{(\frac{1}{7})}{10^5} < \dfrac{(\frac{1}{2})}{10^5}$ and
this means that the approximation is correct to 5 significant digits.

\end{proof}

The elegant calculations in this proof are due to Ingham~\cite{Pars1}.
If we apply the above computations to the case $n = 3$, we obtain the
following results.

\begin{thm}[\textsc{Ingham}]
If $T_0$ denotes the Huygens small-angle approximation to the true
period, $T$, then, for $0 < \alpha < \pi$,
\begin{equation}
\boxed{ \frac{T}{T_0} = \Biggl\{ \frac{2}
{\cos^2\frac{\alpha}{4} + (\cos\frac{\alpha}{2})^{\frac{1}{2}}}
\Biggr\}^2 \cdot \biggl( \frac{1}{1 - R_3} \biggr) }
\end{equation}
and
\begin{equation}
\boxed{ \frac{T}{T_0} = \frac{2}
{\bigl\{1 + (\cos\frac{\alpha}{2})^{\frac{1}{2}} \bigr\}
(\cos\frac{\alpha}{4})^{\frac{1}{2}}
(\cos\frac{\alpha}{2})^{\frac{1}{8}}}
\cdot \biggl( \frac{1}{1 + r_3} \biggr) }
\end{equation}
where
\begin{equation}
\label{Rr3}
0 < r_3 < R_3 < \frac{1}{2^{14}\cos^2\frac{\alpha}{2}}
\Bigl( \sin\frac{\alpha}{4} \tan\frac{\alpha}{4} \Bigr)^8.
\end{equation}
\qed
\end{thm}

\begin{cor}[\textsc{Ingham}]
If $0 < \alpha \leq \frac{\pi}{2}$, then
\begin{equation}
\label{Ing2bis}
0 < r_3 < R_3 < \frac{1}{20000000000}
\end{equation}
and thus the approximation is correct to 10 significant figures in the
worst case, $\alpha = \frac{\pi}{2}$.
\qed
\end{cor}

If we take $\alpha = 179^\circ$ in \eqref{Rr3}, we obtain
$R_3 < 4.50\dots\%$ which is in fair agreement with the machine
calculation of~\cite{CS1} whose machine calculations showed
$R_3 \approx 1\%$. In the next section we show how to bring it into
much closer agreement.


\section{ Complex Multiplication and Renormalization}

As $k$ increases from $0$ to $1$ the quotient $\dfrac{K'(k)}{K(k)}$ decreases monotonically from $+\infty$ to $0$.  Therefore, if $r$ is a positive number, there exists a unique positive number $k$ with $0\leq k<1$ for which the equation\begin{equation}
\label{k1}
\frac{K'(k)}{K(k)}=\sqrt{r}
\end{equation}holds.

The equation \eqref{k1} has two pendulum interpretations.  Let $\alpha+\beta=\pi$ with $0<\beta<\alpha$.

\begin{thm}[classical interpretation] 
The period of a simple pendulum which swings through an angle $2\alpha$ is $\sqrt{r}$ times the period of that same pendulum swinging through an angle $2\beta$.
\qed
\end{thm}

The second \emph{new} interpretation is:

\begin{thm}[new interpretation] 
The period of a simple pendulum which swings through an angle $2\alpha$ is the same as a pendulum $\mathbf{r}$ \textbf{times as long} swinging through an angle $2\beta$.
\qed
\end{thm}

That is, by suitably decreasing the amplitude and simultaneously
increasing the length, one obtains a new pendulum with the same
period. The smaller the amplitude, the more exact is each
approximative formula we have developed. Thus, \emph{if we perform
this process of replacing a given pendulum with pendulums of longer
lengths and smaller amplitudes, our approximations become better and
better}.

In 1811, \textsc{Legendre} \cite{Leg} proved the following remarkable result:
The unique real root $k$ of the equation
\begin{equation}
\label{k2}
\frac{K'(k)}{K(k)}=\sqrt{3}
\end{equation}is \begin{equation}
\label{ }
k := \sin 75^\circ = \frac{\sqrt{6} + \sqrt{2}}{2}
\end{equation} and the complimentary modulus is $k' := \sin 15^\circ = \frac{\sqrt{6} - \sqrt{2}}{2}$.

This equation implies that a pendulum with an amplitude of $300^\circ$
and length~$l$ has \emph{the same period} as a pendulum with an
amplitude of $60^\circ$ and a length~$3l$.

Towards the end of the nineteenth century
\textsc{Greenhill}~\cite{Gre} proved: let the center of the circle of
the pendulum's trajectory be~$O$. Let $B'B$ and $b'b$ be two
horizontal parallel chords of length $\frac{l}{16}$ where $B'B$ is
above the center and subtends the angle $\alpha$ while $b'b$ is below the center and subtends the angle $\beta$ . Let
$k := \sin\frac{1}{2} \alpha$ and
$k' := \sin\frac{1}{2} \beta$. Then
\begin{equation}
\label{Leg1bis}
K = \sqrt{7} \cdot K'.
\end{equation}
This equation implies that a pendulum of length $l$ and amplitude
$\alpha$ has \emph{the same period} as a pendulum of
length $7l$ and amplitude $\beta$.  This is because $2kk'=\frac{1}{8}$.

Both of these examples are instances of an important theorem first stated by \textsc{Abel} in 1828 \cite{Abel} and proved some thirty years later by \textsc{Kronecker} \cite{Kro}, to wit:

\begin{thm}
Let $n$ be a positive integer.  Then the unique positive root $k \equiv k_n$, of the equation \begin{equation}
\label{k3}
\frac{K'(k)}{K(k)}=\sqrt{n}
\end{equation}is the root of a monic algebraic equation with integer coefficients which is solvable by radicals.
\qed
\end{thm}

This means that each \emph{singular modulus}, $k_n$, the root of a transcendental equation, can be explicitly written as a finite combination of radicals, something truly amazing!

Then, 139 years after Abel, \textsc{Selberg} and \textsc{Chowla} \cite{SC} showed that if $k$ is a singular modulus, then $K(k)$, and therefore the period of the pendulum, is expressible in terms of a finite number of \emph{gamma functions} with rational arguments.  In fact, for our two examples,
\begin{equation}
\label{ }
K(k_3)=\frac{3^{\frac{1}{4}}\Gamma^3\left(\frac{1}{3} \right)}{2^{\frac{7}{3}}\pi},\quad\quad\text{and}\quad\quad K(k_7)=\frac{\Gamma\left(\frac{1}{7} \right)\Gamma\left(\frac{2}{7} \right)\Gamma\left(\frac{4}{7} \right)}{7^{\frac{1}{4}}\pi}
\end{equation}

These theorems are a fundamental results in the deep and beautiful theory of
\emph{complex multiplication} of elliptic functions (curves), which is
one of the most active branches of research on the frontiers of modern
mathematics. Unfortunately, it lies outside the scope of our
presentation (see \cite{Cox1} and~\cite{Web}).  Nevertheless, the interpretation of complex multiplication as algebraic relations among the periods is surprising and beautiful and deserves to be better known.

Moreover, these  two examples of replacing one pendulum by another with the same period illustrate what today is called \emph{renormalization} and
it is a fundamental technique in the study of dynamical systems. 
It turns out that the AGM furnishes us with another example of
pendulum renormalization. Let us look at the equation
\begin{equation}
I(a, b) = I(a_1,b_1)
\end{equation}
for the case $a := 1$, $b := \cos\frac{\alpha}{2}$ which gives us
$a_1 = \cos^2\frac{\alpha}{4}$,
$b_1 =  \cos\frac{\alpha}{2})^{\frac{1}{2}}$. We note that
$$
\frac{1}{\sqrt{a_1^2\cos^2\phi' + b_1^2\sin^2\phi'}}
= \frac{1}{a_1} \frac{1}{\sqrt{1 - \frac{b_1^2}{a_1^2}} \sin^2\phi'}
= \frac{1}{\cos^2\frac{\alpha}{4}}
\frac{1}{\sqrt{1 - \tan^4\frac{\alpha}{4} \sin^2\phi'}} \,.
$$
Returning to the formula for the period $T$ of oscillation in an angle
of $4\alpha$, we obtain the equation
\begin{align*}
T = 4\sqrt{\frac{l}{g}} \int_0^{\frac{\pi}{2}}
\frac{d\phi}{\sqrt{1 - \sin^2\frac{\alpha}{2} \sin^2\phi}}
&= 4\sqrt{\frac{\bigl(\frac{l}{\cos^4\frac{\alpha}{4}}\bigr)}{g}}
\int_0^{\frac{\pi}{2}} \frac{d\phi'}
{\sqrt{1 - (\tan^2\frac{\alpha}{4})^2 \sin^2\phi'}}
\\
&=: 4\sqrt{\frac{l_1}{g}} \int_0^{\frac{\pi}{2}}
\frac{d\phi'}{\sqrt{1 - \sin^2\frac{\alpha_1}{2} \sin^2\phi'}} \,.
\end{align*}
This equation can be stated as follows.

\begin{thm}
One iteration of the AGM transforms the pendulum of length $l$ and
maximum angular displacement $\alpha$ into another pendulum with
\textbf{the same period} but with new length
\begin{equation}
\boxed{ l_1 := \dfrac{l}{\cos^4\frac{\alpha}{4}} }
\end{equation}
which is \textbf{longer} than the original length~$l$, and new maximum
angular displacement
\begin{equation}
\boxed{ \alpha_1
:= 2 \cdot \arcsin\bigl( \tan^2\frac{\alpha}{4} \bigr) }
\end{equation}
which is \textbf{smaller} than the original.
\qed
\end{thm}

This theorem allows us to ``explain'' the results of the numerical
investigations of the accuracy of the AGM which \textsc{Carvalhaes}
and \textsc{Suppes} presented in~\cite{CS1}. One uses the AGM to
\emph{renormalize} (reduce) the angular displacement $\alpha$ so that
the \textsc{Ingham} estimates are applicable.

\begin{exmp} 
\textsc{Carvalhaes} and \textsc{Suppes} state that $\dfrac{1}{a_2}$
approximates $\dfrac{T}{T_0}$ to within $1\%$ for
$0 \leq \alpha \leq 163.10^\circ$ while the \textsc{Ingham} bound
\eqref{Inga2} gives $0 \leq \alpha \leq 162.5^\circ$. Again they
report that the approximation has a relative error no bigger than
$\dfrac{1}{2^{52}}$ for $0 \leq \alpha \leq 4.57^\circ$ while the
\textsc{Ingham} bound \eqref{Inga2} gives
$0 \leq \alpha \leq 4.258^\circ$.
\end{exmp}

\begin{exmp} 
As another example, \textsc{Carvalhaes} and \textsc{Suppes} state that
$\dfrac{1}{a_4}$ approximates $\dfrac{T}{T_0}$ to within $1\%$ for
$0 \leq \alpha \leq 179.99^\circ$. The \textsc{Ingham}
bound~\eqref{Rr3} gives that $\dfrac{1}{a_3}$ approximates
$\dfrac{T}{T_0}$ to within $1\%$ for
$0 \leq \alpha \leq 177.98^\circ$. But, one application of the AGM
reduces $\alpha = 179.99^\circ$ to $\alpha = 177.85^\circ$ and now the
\textsc{Ingham} bound \eqref{Rr3} shows that \emph{three more}
applications of the AGM give an approximation to within $1\%$, that
is, $\dfrac{1}{a_4}$ approximates $\dfrac{T}{T_0}$ to within $1\%$ for
$0 \leq \alpha \leq 179.99^\circ$, in agreement with~\cite{CS1}.
\end{exmp}
 
The remaining results in \cite{CS1} can be ``explained'' similarly.
 
We cannot emphasize strongly enough the importance of rigorous upper
and lower bounds for the absolute and relative errors in approximative
formulas. Our analysis allows us to predict ``\emph{a priori}'' the
accuracy of a given approximative formula as well as to justify the
resulting numerical studies. Moreover, \textsc{Ingham}'s elegant and
beautiful investigations give us practical tools to tailor our
approximative formulas to the needs of the accuracy demanded.

\subsubsection*{Acknowledgment}
Support from the Vicerrector\'{\i}a de Investigaci\'on of the 
University of Costa Rica is ack\-now\-ledged.



\end{document}